\documentclass[english,12pt,a4paper]{amsart}
\usepackage[T1]{fontenc}
\usepackage{lmodern}
\usepackage[width=14cm, height=21cm]{geometry}
\usepackage{graphicx}
\usepackage{mathtools}
\usepackage{amssymb}
\usepackage{amsthm}
\usepackage{amsmath}
\usepackage{babel}
\usepackage{hyperref, makecell}
\usepackage{enumerate}

\theoremstyle{plain}
\newtheorem{theorem}{Theorem}[section]

\newtheorem{lemma}{Lemma}[section]
\newtheorem{remark}{Remark}

\begin{document}
\title[Modular equation and Eisenstein series identities]{On a new modular equation of degree five and Eisenstein series identities of associated levels}
	
	\author{Shruthi C. Bhat}
	\address{Shruthi C. Bhat, Manipal Institute of Technology, Manipal Academy of Higher Education, Manipal-576104, Karnataka, India.}
	\email{shruthi.research1@gmail.com}

	\author{B. R. Srivatsa Kumar}
	\address{B. R. Srivatsa Kumar, Manipal Institute of Technology, Manipal Academy of Higher Education, Manipal-576104, Karnataka, India.}
	\email{srivatsa.kumar@manipal.edu}
	
	\maketitle
	
	\begin{abstract}
		In this research article, we obtain few theta function identities of level ten employing Ramanujan's $_1 \psi_1$ summation formula. Using these identities, we derive a new modular equation of degree five.  Further, we establish Eisenstein series identities of level ten using Bailey's very-well poised $_6 \psi_6$ summation formula.
	\end{abstract}


		
	\section{Introduction}
	Throughout the sequel, let $q \in \mathbb{C}$, in such a way that $ |q| < 1$. Then $q$-Pochhammer symbol or $q$-shifted factorial \cite{Berndt1991} is customarily defined as follows: For $b \in \mathbb{C}$ and $m \in \mathbb{N}$,
	\begin{align*}
		(b;q)_m:= \prod_{k=1}^{m}(1-b q^{k-1}), 
	\end{align*}
	and 
	\begin{align*} 
		(b;q)_\infty:= \prod_{k=0}^{\infty}(1-b q^{k}).
	\end{align*} For brevity, the following notation is employed:
	\begin{align*}
		(b_1, b_2, \ldots, b_n;q)_\infty = \prod_{j=1}^{n}(b_j;q)_\infty.
	\end{align*}
	Ramanujan's general theta function \cite[p. 35, Entry 19]{Berndt1991} can be given by
	\begin{align}\label{t3}
		f(r,s)=\sum_{n=-\infty}^{\infty} r^{n(n+1)/2}s^{n(n-1)/2} =(-r, -s, rs; rs)_\infty, \quad |rs|<1.
	\end{align} The following are the theta functions \cite[ p. 36, Entry 22 (i)-(iii)]{Berndt1991} arising from $f(r,s)$:
	\begin{align}
		f(q, q)&= \frac{(-q;-q)_\infty}{(q;-q)_\infty}=\sum_{n=-\infty}^{\infty}q^{n^{2}}=:\varphi(q), \label{t1}\\
		f(q,q^3)&= \frac{(q^2;q^2)_\infty}{(q;q^2)_\infty}= \sum_{n=0}^{\infty} q^{\frac{n(n+1)}{2}}=:	\psi(q)\label{t2}
	\end{align}
	and \begin{align}
		f(-q, -q^2)&=(q;q)_\infty= \sum_{n=-\infty}^{\infty}(-1)^n q^{\frac{n(3n-1)}{2}} =:  f(-q). \label{t4}
	\end{align}
	For brevity, we define $f_n := f(-q^n).$
	Also after Ramanujan, 
	\begin{align}
		\chi(-q) := (q;q^2)_\infty.
	\end{align} 
	\begin{remark}\label{tf} From the above, the following can be obtained with ease.
		\begin{equation*}
			\phi(-q)= \dfrac{f_1^2}{f_2}, \qquad \phi(q)= \dfrac{f_2^5}{f_1^2f_4^2}, \qquad \psi(-q)= \dfrac{f_1 f_4}{f_2}, 
		\end{equation*}
		\begin{equation*}
			\psi(q)= \dfrac{f_2^2}{f_1}, \qquad f(q)=\dfrac{f_2^3}{f_1 f_4}.
		\end{equation*}
	\end{remark}
	
	\noindent $_{s+1}F_s(z)$ is the generalized hypergeometric function \cite{Rainville1971} which is defined by 
	
	\begin{equation}\label{1.1}
		_{s+1}F_s\left[
		\begin{array}{ll}
			\makecell{x_{0}, x_{1}, \ldots,x_{s} \\ y_{1}, y_{2}, \ldots, y_{s}}
			\quad;  z
		\end{array}
		\right]
		=\sum_{m=0}^{\infty}\frac{(x_0)_m(x_1)_m...(x_s)_m }{(y_1)_m(y_2)_m \ldots (y_s)_m}\frac{z^m}{m!}, \qquad |z|<1,
	\end{equation}
	where $(\alpha)_m$ is known as Pochhammer's symbol which is given by
	\begin{equation*}
		(\alpha)_m = \begin{cases}\alpha(\alpha+1)\ldots(\alpha+m-1)&; m\in \mathbb{N} \\
			1&; m=0.
		\end{cases} 
	\end{equation*} This is closely related to gamma function as follows:
	\begin{equation*}
		(\alpha)_m=\frac{\Gamma(\alpha+m)}{\Gamma(\alpha)}.
	\end{equation*}
	In \eqref{1.1}, $x_0,~x_1,~\ldots,~x_s,~y_1,~y_2,~\ldots,~y_s$ are any complex numbers, such that, none of $\{y_j\}_{j=1}^s$ is a non-positive integer. The above series $_{s+1}F_s$ converges absolutely for $|z|< 1$ by D'Alembert ratio test. It is evident that $_{s+1}F_s(z)$ appears in many theoretical and real-world contexts such as theoretical physics, statistics, mathematics, and engineering. 
	
	\noindent The Eisenstein series of level $2k$ is defined as
	\begin{equation*}
		E_{2k}(\tau)= 1 - \frac{4k}{\mathfrak{B}_{2k}} \sum\limits_{n=1}^{\infty} \frac{n^{2k-1}q^n}{1-q^n}
	\end{equation*} where $q \in \mathbb{C}$  with $q=e^{2\pi i \tau},\,\, Im(\tau) >0$ and $B_k$ is the Bernoulli number of $k^{th}$ order defined by
	\begin{equation*}
		\frac{x}{e^x - 1} = \sum \limits_{k=0}^{\infty} \frac{\mathfrak{B}_k x^k}{k!}.
	\end{equation*} Note that 	$\mathfrak{B}_0 = 1$, $\mathfrak{B}_1=-1/2$, $\mathfrak{B}_2=1/6$, $\mathfrak{B}_3=0$, $\mathfrak{B}_4 = -1/30$, $\ldots$ $\mathfrak{B}_{2k+1}=0$, for $k\geq 1$, $\mathfrak{B}_k \in \mathbb{Q}$.	$L(q)$, $M(q)$ and $N(q)$ are the famous Ramanujan's Eisenstein series that are defined by
	\begin{equation*}\label{P}
		L(q):=	E_2(\tau)  = 1- 24 \sum \limits_{m=1}^{\infty} \frac{m q^m}{1-q^m},
	\end{equation*}
	\begin{equation*}
		M(q) :=	E_4(\tau) =  1+ 240 \sum \limits_{m=1}^{\infty} \frac{m^3 q^m}{1-q^m}
	\end{equation*}
	and
	\begin{equation*}
		N(q) := 	E_6(\tau) = 1-504 \sum \limits_{m=1}^{\infty} \frac{m^5 q^m}{1-q^m}.
	\end{equation*}
	For brevity, we define
	\begin{equation}\label{ESD}
		L_n:=L(q^n) = 1- 24 \sum \limits_{m=1}^{\infty} \frac{m q^{nm}}{1-q^{nm}}.
	\end{equation} For details on Eisenstein series one may refer \cite{Apostol1976, Rankin1977}. Ramanujan's contributions \cite{Ramanujan1916}, \cite[pp. 136--162]{Ramanujan2000} to the theory of Eisenstein series is significant, which has applications in the theory of  the partition function, the divisor functions, the number of different ways of expressing a natural number irrespective of the order of summands and series representation for $1/\pi$.
	We define modular equation here.  An equation relating $\gamma$ and $\delta$ induced by 
	\begin{equation*}
		n\frac{_2F_1\left(\dfrac{1}{2}, \dfrac{1}{2}; 1; 1-\gamma\right)}{_2F_1\left(\dfrac{1}{2}, \dfrac{1}{2}; 1; \gamma\right)}=\frac{_2F_1\left(\dfrac{1}{2}, \dfrac{1}{2}; 1; 1-\delta\right)}{_2F_1\left(\dfrac{1}{2}, \dfrac{1}{2}; 1; \delta\right)},
	\end{equation*} is known as the modular equation of degree $n$ in the theory of signature two. In the above relation, $\delta$ carries $n^{\text{th}}$ degree over
	$\gamma$, where
	$	m:=\frac{z_1}{z_n}$
	is the multiplier, 
	and $z_1=~_2F_1\left(\dfrac{1}{2}, \dfrac{1}{2}; 1; \gamma\right)$ and $z_n=~_2F_1\left(\dfrac{1}{2}, \dfrac{1}{2}; 1; \delta\right)$.
	
	The first known modular equation  is due to Legendre \cite{Legendre1825}, that is stated as follows:
	If $\beta$ has degree 3 over $\alpha$, then
	\begin{align*}
		(\alpha \beta)^{1/4} + {(1-\alpha) (1-\beta)}^{1/4} =1.
	\end{align*}
	
	\noindent Jacobi \cite{Jacobi1829} discovered modular equations of degree 3 and 5. Later, following the works of Legendre and Jacobi, many mathematicians  contributed for the theory of modular equation. For more details one may refer \cite{Berndt2000, Hanna1928, Ramanujan1914}.
	
	\noindent   Bailey's \cite{Bailey1936} very-well poised $_6 \psi_6$ summation formula is given by
	\begin{align}\label{B6p6}
		\sum_{k=-\infty}^{\infty} &\dfrac{(qP^{1/2};q)_k (-qP^{1/2};q)_k (Q;q)_k (R;q)_k (S;q)_k (T;q)_k}{(P^{1/2};q)_k (-P^{1/2};q)_k (Pq/Q;q)_k (Pq/R;q)_k (Pq/S;q)_k (Pq/T;q)_k} \left(\dfrac{qP^2}{QRST}\right)^k  \nonumber \\
		& = \dfrac{(Pq, Pq/QR, Pq/QS, Pq/QT, Pq/RS, Pq/RT, Pq/ST, q, q/P;q)_\infty}{(Pq/Q, Pq/R, Pq/S, Pq/T, q/Q, q/R, q/S, q/T, qP^2/QRST;q)_\infty}.
	\end{align}
	Ramanujan documented numerous modular equations of various degrees.  In his Notebook \cite{Berndt1991}, Chapter 19 deals with the modular equations of degrees 3, 5 and 7 and also the theta function identities associated to those. According to Berndt \cite{Berndt1991}, for each particular degree, Ramanujan appeared to first have derived a set of wonderful identities consisting of theta functions of certain arguments which were then used to obtain a fascinating set of modular equations of the corresponding degree (order). But, Berndt appeared to have reversed this custom and proved theta function identities employing modular equations. Berndt used to employ three different methods to  establish modular equations documented by Ramanujan, first was by the Ramanujan's theory of theta functions, second by the method of parameterization and third by the theory of modular forms. Shen \cite{Shen1994} and Baruah and Barman \cite{Baruah_2006} used the concept of theta functions to establish modular equations of degree three documented by Ramanujan. Shen \cite{Shen1995}, Liu \cite{Liu2001} and Baruah et al. \cite{Baruah_2012} also used the theory of theta functions to obtain the proofs of modular equations of degree five recorded by Ramanujan. Liu \cite{Liu2010} obtained certain nice theta function identities associated to modular equations of degree seven, by employing the complex variable theory of elliptic functions. Vasuki and Veeresha \cite{Vasuki_2015} established modular equations of degree seven using theory of theta functions. Barman and Baruah \cite{Barman_2010} obtained alternate proofs of Ramanujan's theta function identities related to the modular equations of composite degrees three, five and fifteen. They also gave the elementary proofs of certain modular equations of degree fifteen. Recently, Vasuki and Yathirajsharma \cite{Vasuki_2024} reproved a set of modular equations of degree eleven by employing elementary algebraic techniques and the method of parametrization. This was the first set of modular equations, the proofs of which by Berndt \cite{Berndt1991} moved to theory of modular forms from elementary algebraic procedures. Modular equations of degrees eleven, thirteen, seventeen, nineteen, twenty three, thirty one, forty seven, and seventy one were established in Chapter 20 of Ramanujan's Notebook \cite{Berndt1991}, which also involved the study of modular equations of composite degree or ``mixed'' modular equations. The modular equations of composite degrees 3 and 9 were documented in Chapter 20, which were established by Berndt with the help of parametrization. The modular equations of various degrees and theta function identities of various levels are used to obtain various beautiful Eisenstein series identities. Wonderful work on this context can be found in the literature. Liu \cite{Liu2003} established few Eisenstein series identities using Ramanujan's modular equations of degree seven, through complex variable theory of elliptic functions. Cooper and Ye \cite{Cooper_2016} established Eisenstein series identities of level fourteen. Vasuki and Veeresha \cite{Vasuki_2016} established Ramanujan's Eisenstein series of level seven and fourteen through elementary approach by the help of $_6 \psi _6$ summation formula.
	
	Motivated by the work in literature, we derive theta function identities of level ten using Ramanujan's $_1 \psi_1$ summation formula. Utilizing these, we establish modular equation of degree five. Also, we obtain Eisenstein series of associated levels with the help of Bailey's very-well poised $_6 \psi_6$ summation formula.
	\section{Preliminaries}
	We list some basic results   in this section, that are needed in proving our main results.
	\begin{lemma}\label{LemmaA}
		If $xy=zw$, then we have the following:
		\begin{enumerate}[(i)]
			\item $	f(x,y) f(z,w) + f(-x,-y) f(-z,-w) = 2f(xz,yw) f(xw, yz).$
			\item $f(x,y) f(z,w) - f(-x,-y) f(-z,-w) = 2x f(y/z,xz^2w) f(y/w, xzw^2).$
			\item $f(x,y) f(z,w) f(xn,y/n) f(zn,w/n)-f(-x,-y) f(-z,-w)f(-xn,-y/n)\\ \times f(-zn,-w/n)
			= 2x f(z/x,xw) f(w/xn,xzn) f(n,xy/n) \psi(xy).$			
		\end{enumerate}
	\end{lemma}
	\begin{proof}
		For the proofs of (i) and (ii), one can refer Ramanujan Notebook \cite[pp. 45-47]{Berndt1991}and \cite[p. 199]{Ramanujan1957}. (iii) is due to Adiga et al. \cite{Adiga_1985}.
	\end{proof}

	\begin{lemma}
		We have 
		\begin{align}\label{2.1}
			\frac{\varphi(q^5)}{\varphi(-q^5)} + \dfrac{f(-x,-y)}{f(x,y)} = \dfrac{2 f^2(xq^5, yq^5)}{\varphi(-q^5) f(x,y)},
		\end{align} 
		\begin{align}\label{2.2}
			\frac{\varphi(q^5)}{\varphi(-q^5)} - \dfrac{f(-x,-y)}{f(x,y)} = \dfrac{2x f^2(y/q^5, xq^{15})}{\varphi(-q^5) f(x,y)},
		\end{align} 
		\begin{align}\label{2.3}
			\frac{2x\psi(x^2y^2)}{\varphi(xy)} - \dfrac{f(x^2,y^2)}{f(y/x,x^3y)} = \dfrac{ f^2(x, y)}{\varphi(xy) f(y/x,x^3y)}.
		\end{align} and 
		\begin{align}\label{2.4}
			&	\dfrac{f(x,y) f(z,w) f(xn,y/n) f(zn,w/n)}{f(-x,-y) f(-z,-w)f(-xn,-y/n) f(-zn,-w/n)} -1 \nonumber \\ &= \dfrac{2x f(z/x,xw) f(w/xn,xzn) f(n,xy/n) \psi(xy)}{f(-x,-y) f(-z,-w)f(-xn,-y/n) f(-zn,-w/n)}.
		\end{align} 
	\end{lemma}
	\begin{proof}
		Substituting $z=w=q^5$ in the equations (i) and (ii) of Lemma \ref{LemmaA}, we obtain the above two identities \eqref{2.1} and \eqref{2.2} respectively. Adding the equations (i) and (ii) of Lemma \ref{LemmaA} and substituting $z=x$ and $w=y$, we obtain \eqref{2.3}. Substituting $xy=zw$ in (iii) of Lemma \ref{LemmaA}, we obtain \eqref{2.4}.
	\end{proof}

	\begin{lemma}\label{LemmaB}
		We have
		\begin{enumerate}[(i)]
			\item   $f(q,q^4) f(q^2,q^3) = \dfrac{\varphi(-q^5) f_5}{\chi(-q)}.$
			\item $f(-q,-q^4) f(-q^2,-q^3) = f_1f_5.$
			\item $f(q,q^{9}) f(q^3,q^{7})  = \chi(q) f_5 f_{20}.$  
		\end{enumerate}
	\end{lemma}
	\begin{proof}
		The proofs of the above-mentioned identities can be found in \cite[Entry 9, p. 258]{Berndt1991}.
	\end{proof}
	
	\noindent Ramanujan's $_1 \psi _1$ summation formula \cite{Adiga_1985, Andrews1969} is given by 
	\begin{align}\label{1sum}
		_1 \psi _1 \left( \substack{a \\ b} ; z\right) = \sum_{n= -\infty}^{\infty} \dfrac{(a;q)_n}{(b;q)_n} z^n = \dfrac{(az;q)_\infty (b/a;q)_\infty (q/az;q)_\infty (q;q)_\infty}{(z;q)_\infty (q/a;q)_\infty (b/az;q)_\infty (b;q)_\infty},
	\end{align} 
	where $\dfrac{b}{a} < |z|< 1$.
	
	\noindent For convenience, we write
	\begin{align}\label{Gamma}
		\Gamma_1 := \dfrac{f(q, q^{9})}{f(-q, -q^{9})} , \,\,\, \Gamma_2 := \dfrac{f(q^3, q^{7})}{f(-q^3, -q^{7})}
	\end{align} and 
	\begin{align}\label{Omega}
		\Omega_1 := \dfrac{f(q^{4}, q^{6})}{f(-q, -q^{9})}, \,\,\, \Omega_2 := q\dfrac{f(q^2, q^{8})}{f(-q^3, -q^{7})}.
	\end{align}

	\section{Theta function identities of level ten}
	We obtain theta function identities of level 10 in this section.

	\begin{lemma}\label{LF}
		We have 
		\begin{align*}
			\Omega_1 -\Omega_2  = \dfrac{\psi^2(q^2)- q^2\psi^2(q^{10})}{\psi^2(-q^5)},
		\end{align*} where $\Omega_1$ and $\Omega_2$ are defined as in \eqref{Omega}.
	\end{lemma}
	
	\begin{proof}
		We have \cite{Berndt1991}
		\begin{align*}
			\psi^2(q^2) =  \sum_{k=0}^{\infty} \dfrac{ q^{k}}{1+q^{2k+1}}.
		\end{align*}
		Hence,
		\begin{align}\label{3.2a}
			\psi^2(q^2)- q^2\psi^2(q^{10}) &=  \sum_{k=0}^{\infty} \dfrac{q^{k}}{1+q^{2k+1}} - q^3\sum_{k=0}^{\infty} \dfrac{ q^{5k}}{1+q^{10k+5}}. \nonumber \\
			&= \sum_{r=0}^{1}  \sum_{k=- \infty} ^{\infty} \dfrac{q^{5k+r}}{1+q^{10k+2r+1}}.
		\end{align}
		Substituting $b=aq$ in \eqref{1sum}, we obtain
		\begin{align*}
			\sum_{n= -\infty}^{\infty} \dfrac{z^n}{1-aq^{n}} =  \dfrac{f(-az,-q/az) (q;q)_\infty^3}{f(-z,-q/z) f(-a,-q/a)}.
		\end{align*}
		Replacing $q$ by $q^{10}$ in the above, we obtain
		\begin{align}\label{1sum2_10}
			\sum_{n= -\infty}^{\infty} \dfrac{z^n}{1-aq^{10n}} = \dfrac{f(-az,-q^{10}/az) (q^{10};q^{10})_\infty^3}{f(-z,-q^{10}/z) f(-a,-q^{10}/a)}.
		\end{align}
		Substituting $z=q^{5}$ and $a=-q$ and $-q^{3}$ in \eqref{1sum2_10}, the following identities are obtained respectively:
		\begin{align*}
			\sum_{n= -\infty}^{\infty} \dfrac{q^{5n}}{1+q^{10n+1}} = \dfrac{f(q^4,q^{6}) (q^{10};q^{10})_\infty^3}{\varphi(-q^5)f(q,q^{9}) },
		\end{align*} and
		\begin{align*}
			\sum_{n= -\infty}^{\infty} \dfrac{q^{5n}}{1+q^{10n+3}} = \dfrac{f(q^2,q^{8}) (q^{10};q^{10})_\infty^3}{\varphi(-q^5)f(q^3,q^{7}) }.
		\end{align*}
		Using the above identities in \eqref{3.2a} and changing $q$ to $-q$, upon using the definitions of $\Omega_1$ and $\Omega_2$ , we obtain the desired result. 
	\end{proof}

	\begin{theorem}\label{TD}
		We have 
		\begin{align*}
			\Gamma_1^2 + \Gamma_2^2 =& \dfrac{\varphi^4(q)}{4 \varphi^4(-q^{10})} + \dfrac{\varphi^4(q^5)}{4\varphi^4(-q^{10})} - \dfrac{\varphi^2(q) \varphi^2(q^5)}{2 \varphi^4(-q^{10})}  + \dfrac{2 \psi(q) \psi(-q^5) } {\psi(-q) \psi(q^5) },
		\end{align*}  where $\Gamma_1$ and $\Gamma_2$ are defined as in \eqref{Gamma}.
	\end{theorem}
	
	\begin{proof} We have
		\begin{align}\label{G1}
			\Gamma_1^2 + \Gamma_2^2 =	(\Gamma_1-\Gamma_2)^2+ 2\Gamma_1 \Gamma_2.
		\end{align}
		Using the definitions, we obtain 
		\begin{align}\label{G2}
			\Gamma_1\Gamma_2 = \dfrac{f(q,q^9) f(q^3,q^7)}{f(-q,-q^9)f(-q^3,-q^7)}=\dfrac{ \psi(q) \psi(-q^5) } {\psi(-q) \psi(q^5) }.
		\end{align} 
		From (3.19) of \cite{Harshitha_2022}, we have the following identity:
		\begin{align*}
			\Gamma_1 -\Gamma_2  = \dfrac{\varphi^2(q)+ \varphi^2(q^5)}{2 \varphi^2(-q^{10})},
		\end{align*} where $\Gamma_1$ and $\Gamma_2$ are defined as in \eqref{Gamma}.
		Using the above equation and \eqref{G2} in \eqref{G1}, the result is obtained.
		
	\end{proof}
	
	\begin{theorem}\label{TF}
		We have 
		\begin{align*}
			\Omega_1^2 + \Omega_2^2 =& \dfrac{\psi^4(q^2)}{ \psi^4(-q^{5})} + \dfrac{q^4\psi^4(q^{10})}{\psi^4(-q^{5})} - \dfrac{2q^2\psi^2(q^2) \psi^2(q^{10})}{ \psi^4(-q^{5})}   		+ \dfrac{ 2q\varphi(-q^{10}) \psi(-q) } {\varphi(-q)\psi(q^5) } ,
		\end{align*}		where $\Omega_1$ and $\Omega_2$ are defined as in \eqref{Omega}.
	\end{theorem}
	
	\begin{proof}We have 
		\begin{align*}
			\Omega_1^2 + \Omega_2^2 =	(\Omega_1-\Omega_2)^2+ 2\Omega_1 \Omega_2.
		\end{align*}
		Using the definitions, we obtain 
		\begin{align*}
			\Omega_1\Omega_2 = \dfrac{qf(q^4,q^6) f(q^2,q^8)}{f(-q,-q^9)f(-q^3,-q^7)}=\dfrac{ q\varphi(-q^{10}) \psi(-q) } {\varphi(-q) \psi(q^5) }.
		\end{align*} 
		Using this and Lemma \ref{LF} in the expression above, the result is obtained.
		
	\end{proof}
	
	\begin{theorem}\label{T/F}
		We have 
		\begin{align*}
			\dfrac{1}{\Omega_1^2} +\dfrac{1}{\Omega_2^2} =& \dfrac{\varphi^4(-q)}{16q^4 \psi^4(-q^{5})} + \dfrac{\varphi^4(-q^5)}{16q^4\psi^4(-q^{5})} - \dfrac{\varphi^2(-q) \varphi^2(-q^5)}{8q^4\psi^4(-q^{5})} + \dfrac{2 \varphi(-q)\psi(q^5)} {q \varphi(-q^{10})\psi(-q)},
		\end{align*} 	where $\Omega_1$ and $\Omega_2$ are defined as in \eqref{Omega}.
	\end{theorem}
	
	\begin{proof}
		We have 
		\begin{align}\label{O1}
			\dfrac{1}{	\Omega_1^2} + \dfrac{1}{\Omega_2^2} =	\left(\dfrac{1}{\Omega_1}-\dfrac{1}{\Omega_2}\right)^2+ \dfrac{2}{\Omega_1 \Omega_2}.
		\end{align}
		From \cite{Bhat_2025},	we have 
		\begin{align*}
			\sum_{j=1}^{\frac{(k-1)}{2}} \dfrac{(-1)^{j-1}}{q^{j-1}} \dfrac{f(-q^{2j-1}, -q^{2k-2j+1})}{f(q^{k-2j+1}, q^{k+2j-1})} = \dfrac{\varphi^2(-q)-\varphi^2(-q^k)}{4q^{\frac{k-1}{2}} \psi^2(-q^k)},	
		\end{align*} for any $k \in \mathbb{N}.$
		Substituting $k=5$ in the above and using the resultant expression along with expression for $	\Omega_1\Omega_2$  in \eqref{O1}, the result is obtained.
	\end{proof}

	\section{Modular equations of degree five}
	Here, we establish modular equation of degree 5, which is new and devoid of symmetry as compared to modular equations given by Ramanujan.
	\begin{theorem}
		We have the following modular equation of degree 5:
		\begin{align*}
			&m^2 + 1 -2m +8\left(\dfrac{(1-\beta)^5}{(1-\alpha)}\right)^{1/8} - 8(1-\alpha)^{1/2} (1-\beta)^{1/2} \\ &= 	m^2 \alpha + \beta  	 - 	2m^2 \alpha^{1/2} \beta^{1/2}+ 8\left(\dfrac{\alpha \beta^3 (1-\beta)^5}{(1-\alpha)}\right)^{1/8},
		\end{align*} where $\beta$ has degree 5 over $\alpha$ and $m$ is the multiplier. 
	\end{theorem}
	
	\begin{proof}
		Substituting $z=w=q^5$ and $n=1$ in (iii) of Lemma \ref{LemmaA}, we obtain
		\begin{align}\label{T3.4.1}
			\dfrac{f^2(x,y)}{f^2(-x,-y)} - \dfrac{\varphi^2(-q^5)}{\varphi^2(q^5)} = \dfrac{4x \psi^2(xy) f^2(q^5/x,xq^5)}{\varphi^2(q^5) f^2(-x,-y)}.
		\end{align}
		Replacing $(x,y)$ by $(q,q^{9})$ and $(q^3,q^{7})$  in \eqref{T3.4.1} and adding the two  resultant identities, we obtain
		\begin{align*}
			\left(	\Gamma_1^2 + \Gamma_2^2 \right) -2 \dfrac{\varphi^2(-q^5)}{\varphi^2(q^5)} =  4q \dfrac{\psi^2(q^{10})}{\varphi^2(q^5)} \left( 	\Omega_1^2 + \Omega_2^2\right).
		\end{align*}
		Using Theorem \ref{TD} and Theorem \ref{TF} in the expression above, we deduce
		\begin{align*}
			&\dfrac{\varphi^4(q)}{4 \varphi^4(-q^{10})} + \dfrac{\varphi^4(q^5)}{4\varphi^4(-q^{10})} - \dfrac{\varphi^2(q) \varphi^2(q^5)}{2 \varphi^4(-q^{10})}  + \dfrac{2 \psi(q) \psi(-q^5) } {\psi(-q) \psi(q^5) }-2 \dfrac{\varphi^2(-q^5)}{\varphi^2(q^5)} \\
			&=  4q \dfrac{\psi^2(q^{10})}{\varphi^2(q^5)} \left( \dfrac{\psi^4(q^2)}{ \psi^4(-q^{5})} + \dfrac{q^4\psi^4(q^{10})}{\psi^4(-q^{5})} - \dfrac{2q^2\psi^2(q^2) \psi^2(q^{10})}{ \psi^4(-q^{5})}   		+ \dfrac{ 2q\varphi(-q^{10}) \psi(-q) } {\varphi(-q)\psi(q^5) } \right).
		\end{align*}
		Using the definitions of $\psi$ and $\varphi$ in the above, on simplification, the desired result is accomplished. 
	\end{proof}

	\section{Eisenstein series identities of level ten and twenty}
	In this section, we establish the Eisenstein series identities of level 10 and 20,  using the  Bailey's  very-well poised $_6 \psi_6$ summation formula.
	Putting $P=xy$, $Q=R=x$ and $S=T=y$ in \eqref{B6p6}, one can easily deduce
	\begin{align}\label{B6p6*}
		\sum_{k=-\infty}^{\infty} & \dfrac{(\sqrt{xy} q;q)_k (-\sqrt{xy} q;q)_k (x;q)_k^2 (y;q)_k^2}{(\sqrt{xy} ;q)_k (-\sqrt{xy} ;q)_k (xq;q)_k^2 (yq;q)_k^2} q^k \nonumber \\ &=\dfrac{(xyq;q)_\infty (yq/x;q)_\infty (xq/y;q)_\infty (q/xy;q)_\infty (q;q)_\infty^4 }{(yq;q)_\infty^2 (xq;q)_\infty^2 (q/x;q)_\infty^2 (q/y;q)_\infty^2} .
	\end{align}
	From  (4.5) of \cite[p. 197]{Agarwal1996}, one can easily deduce
	\begin{align}\label{RPA}
		\sum_{k=-\infty}^{\infty} &\left[\dfrac{xq^k}{(1-xq^k)^2} - \dfrac{yq^k}{(1-yq^k)^2}\right] \nonumber \\ &= \dfrac{x(1-y/x) (1-xy)}{(1-x)^2 (1-y)^2} \sum_{k=-\infty}^{\infty} \dfrac{(\sqrt{xy} q;q)_k (-\sqrt{xy} q;q)_k (x;q)_k^2 (y;q)_k^2}{(\sqrt{xy} ;q)_k (-\sqrt{xy} ;q)_k (xq;q)_k^2 (yq;q)_k^2} q^k.
	\end{align}
	Using \eqref{RPA} in the above and changing $q$ to $q^{10}$, we deduce
	\begin{align}\label{4.1.1}
		\sum_{n=-\infty}^{\infty} \left[\dfrac{xq^{10n}}{(1-xq^{10n})^2} - \dfrac{yq^{10n}}{(1-yq^{10n})^2}\right] =x \dfrac{f(-xy, -q^{10}/xy) f(-y/x,-xq^{10}/y) (q^{10};q^{10})_\infty^6}{f^2(-x, -q^{10}/x) f^2(-y,-q^{10}/y)}.
	\end{align}
	
	\begin{theorem} 
		The following Eisenstein series identity holds.
		\begin{align*}
			\dfrac{1}{24} \left[-L_1 +L_2 +5L_5 -5L_{10}\right] =& -\dfrac{\varphi^4(q^5)}{2} + \dfrac{\varphi^4(-q^{10})}{4}\left(\dfrac{\varphi^4(q)}{4\varphi^4(-q^{10})} \right. \\ & \left. + \dfrac{\varphi^4(q^5)}{4\varphi^4(-q^{10})} -\dfrac{\varphi^2(q) \varphi^2(q^5)}{4\varphi^4(-q^{10})} +\dfrac{2 \psi(q) \psi(-q^5)}{\psi(-q) \psi(q^5)}\right),
		\end{align*}
		where $L_n$ is the Eisenstein series defined as in \eqref{ESD}.
	\end{theorem}
	
	\begin{proof}
		Replacing $(x,y)$ by $(q,-1)$ and $(q^3,-1)$ in \eqref{4.1.1}, and adding the two  resultant identities and using the definitions of $\Gamma_1$ and $\Gamma_2$, we obtain
		\begin{align}\label{T4.1.1}
			\sum_{n=0}^{\infty} \left[\dfrac{q^{2n+1}}{(1-q^{2n+1})^2} - \dfrac{q^{10n+5}}{(1-q^{10n+5})^2}\right] &+ 2\sum_{n=-\infty}^{\infty} \dfrac{q^{10n}}{(1+q^{10n})^2} \nonumber \\
			&= \dfrac{\varphi^4(-q^{10})}{4} \left(\Gamma_1^2 + \Gamma_2^2\right).
		\end{align}
		Substituting $x=q^5$ and $y=-1$ in \eqref{4.1.1}, we obtain
		\begin{align}\label{T4.1.2}
			\sum_{n=-\infty}^{\infty} \left[\dfrac{q^{10n+5}}{(1-q^{10n+5})^2} + \dfrac{q^{10n}}{(1+q^{10n})^2}\right] =\dfrac{\varphi^4(q^{5})}{4}.
		\end{align}
		Using \eqref{T4.1.2} in \eqref{T4.1.1}, we obtain
		\begin{align*}
			\sum_{n=0}^{\infty} \dfrac{q^{2n+1}}{(1-q^{2n+1})^2} -5 \sum_{n=0}^{\infty} \dfrac{q^{10n+5}}{(1-q^{10n+5})^2}= -\dfrac{\varphi^4(q^5)}{2} + \dfrac{\varphi^4(-q^{10})}{4} \left(\Gamma_1^2 + \Gamma_2^2\right).
		\end{align*}
		Expanding the summation on the left hand side of the expression above and grouping the like terms suitably, we arrive at
		\begin{align*}
			\sum_{n=1}^{\infty} \dfrac{nq^{n}}{(1-q^n)} &- \sum_{n=1}^{\infty} \dfrac{nq^{2n}}{(1-q^{2n})}-5 \sum_{n=1}^{\infty} \dfrac{q^{5n}}{(1-q^{5n})}+ 5\sum_{n=1}^{\infty} \dfrac{q^{10n}}{(1-q^{10n})}\\
			&= -\dfrac{\varphi^4(q^5)}{2} + \dfrac{\varphi^4(-q^{10})}{4} \left(\Gamma_1^2 + \Gamma_2^2\right).
		\end{align*}
		Using the definition of $L_n$ in the above, we obtain
		\begin{align*}
			\dfrac{1}{24} \left[-L_1 +L_2 +5L_5 -5L_{10}\right] = -\dfrac{\varphi^4(q^5)}{2} + \dfrac{\varphi^4(-q^{10})}{4} \left(\Gamma_1^2 + \Gamma_2^2\right).
		\end{align*}
		Using Theorem \ref{TD} in the expression atop, the desired result is accomplished on simplification.
	\end{proof}
	
	\begin{theorem}
		The following Eisenstein series identity holds.
		\begin{align*}
			\dfrac{2}{3} \left[L_2 -4L_4 -5L_{10} +20L_{20}\right] =& 8 -8\varphi^4(-q^5) -16q^4 \psi^4(-q^5) \left(\dfrac{\varphi^4(-q)}{16q^4 \psi^4(-q^5)} \right. \\ & \left.+ \dfrac{\varphi^4(-q^5)}{\psi^4(-q^5)} - \dfrac{\varphi^2(-q) \varphi^2(-q^5)}{8q^4 \psi^4(-q^5)} + \dfrac{2 \varphi(-q) \psi(q^5)}{q\varphi(-q^{10}) \psi(-q) }\right),
		\end{align*}
		where $L_n$ is the Eisenstein series defined as in \eqref{ESD}.
	\end{theorem}
	
	\begin{proof}
		Replacing $(x,y)$ by $(-q^2,-q^5)$ and $(-q^4,-q^5)$ in \eqref{4.1.1}, and adding the  four resultant identities and using the definition of $\Omega_1$ and $\Omega_2$, we obtain
		\begin{align}\label{T4.2.1}
			\sum_{n=0}^{\infty} \left[-\dfrac{q^{2n+2}}{(1+q^{2n+2})^2} + \dfrac{q^{10n+10}}{(1+q^{10n+10})^2}\right] &+ 2\sum_{n=-\infty}^{\infty} \dfrac{q^{10n+5}}{(1+q^{10n+5})^2} \nonumber \\
			&=-q^4 \psi^4(-q^{5}) \left(\dfrac{1}{\Omega_1^2} -\dfrac{1}{\Omega_2^2} \right).
		\end{align}
		Substituting $x=-1$ and $y=-q^5$ in \eqref{4.1.1}, we obtain
		\begin{align}\label{T4.2.2}
			\sum_{n=-\infty}^{\infty} \left[-\dfrac{q^{10n}}{(1+q^{10n})^2} + \dfrac{q^{10n+5}}{(1+q^{10n+5})^2}\right] =\dfrac{\varphi^4(-q^{5})}{4}.
		\end{align}
		Using \eqref{T4.2.2} in \eqref{T4.2.1}, we obtain
		\begin{align}\label{T4.2.3}
			-\sum_{n=0}^{\infty} \dfrac{q^{2n+2}}{(1+q^{2n+2})^2} &+5 \sum_{n=0}^{\infty} \dfrac{q^{10n+10}}{(1+q^{10n+10})^2} \nonumber\\
			&= -\dfrac{\varphi^4(-q^5)}{2} -q^4 \psi^4(-q^5)  \left(\dfrac{1}{\Omega_1^2} +\dfrac{1}{\Omega_2^2} \right).
		\end{align}
		We know that
		\begin{align}\label{T4.2.4}
			\sum_{n=1}^{\infty} \dfrac{q^n}{(1+q^n)^2} = \sum_{n=1}^{\infty} \dfrac{nq^n}{(1-q^n)} - 4 \sum_{n=1}^{\infty} \dfrac{nq^{2n}}{(1-q^{2n})}. 
		\end{align}
		Expanding the summation on the left hand side of \eqref{T4.2.3}, grouping the like terms suitably and using \eqref{T4.2.4}, we obtain
		\begin{align*}
			-\sum_{n=1}^{\infty} \dfrac{nq^{2n}}{(1-q^{2n})} &+ 4\sum_{n=1}^{\infty} \dfrac{nq^{4n}}{(1-q^{4n})}+9 \sum_{n=1}^{\infty} \dfrac{q^{10n}}{(1-q^{10n})}- 20\sum_{n=1}^{\infty} \dfrac{q^{20n}}{(1-q^{20n})}\\
			&= -\dfrac{\varphi^4(-q^5)}{2} -q^4 \psi^4(-q^5)  \left(\dfrac{1}{\Omega_1^2} +\dfrac{1}{\Omega_2^2} \right).
		\end{align*}
		Using the definition of $L_n$ in the above, we obtain
		\begin{align*}
			\dfrac{1}{24} \left[L_2 -4L_4 -5L_{10} +20L_{20}\right] =\dfrac{1}{2}-\dfrac{\varphi^4(-q^5)}{2}  -q^4 \psi^4(-q^5)  \left(\dfrac{1}{\Omega_1^2} +\dfrac{1}{\Omega_2^2} \right).
		\end{align*}
		Using Theorem \ref{T/F} in the expression atop, the desired result is accomplished on simplification.
	\end{proof}
	
	\section{Conclusion}
	We obtained theta function identities of level ten using Ramanujan's $_1 \psi_1$ summation formula. These identities are then utilized to establish a modular equation of degree five. Further, we established Eisenstein series identities of level ten and twenty using Bailey's very-well poised $_6 \psi_6$ summation formula.

	\section*{Acknowledgements}
	The first author acknowledges the support of DST-INSPIRE, Department of Science and Technology, Government of India, India for providing INSPIRE fellowship [DST/INSPIRE/03/2022/004970].
	%%%%%%%%%%%%%%%%%%%%%%%%%%%%%%%%%%%%%%%%%%%%%%%%%%%%%%%%%%%%%%%%

\end{document}